\documentclass[11pt]{amsart}
\usepackage{amsmath, amsthm, amssymb,xspace}
\usepackage[breaklinks=true]{hyperref}
\usepackage{amsmath,amscd}

\theoremstyle{plain}
\newtheorem{theorem}[equation]{Theorem}
\newtheorem{lemma}[equation]{Lemma}
\newtheorem{prop}[equation]{Proposition}
\newtheorem{cor}[equation]{Corollary}

\theoremstyle{definition}

\newcommand{\bggo}{\mathcal O}

\newcommand{\mf}[1]{\displaystyle{\mathfrak{#1}}}

\newcommand{\comment}[1]{}

\DeclareMathOperator{\spec}{\ensuremath{Spec}}

\DeclareMathOperator{\Gr}{\ensuremath{gr}}

\DeclareMathOperator{\Hom}{\ensuremath{Hom}}

\begin{document}

\title[Hochschild cohomology of deformation quantizations]{Hochschild cohomology of deformation quantizations over $\mathbb{Z}/p^n\mathbb{Z}$}
\author{Akaki Tikaradze}
\address{The University of Toledo, Department of Mathematics, Toledo, Ohio, USA}
\email{\tt tikar06@gmail.com}
\maketitle

\begin{abstract}

Let $A_1$ be an Azumaya algebra over a smooth affine symplectic variety $X$ over $\spec F_p,$
where $p$ is an odd prime. Let $A$ be a deformation quantization of $A_1$ over the $p$-adic integers.
In this note we show that for all $n\geq 1,$ the Hochschild cohomology of $A/p^nA$ is isomorphic
to the de Rham-Witt complex $W_n\Omega^{*}_X$ of $X $ over $\mathbb{Z}/p^n\mathbb{Z}.$ 
We also compute the center of deformations of certain affine Poisson varieties over $F_p.$ 

\end{abstract}

\vspace*{0.5in}

Let $\bf{k}$ be a perfect field of characteristic $p>2.$
 For $n\geq 1$, let $W_n(\bold{k})$ denote  the ring of length $n$ Witt vectors over $\bf{k}.$
Also, $W(\bold{k})$ will denote the ring of Witt vectors over $\bf{k}$. 
 Let $X$ be an affine smooth symplectic variety over $\bf{k}.$ Let  $\lbrace, \rbrace$ denote the corresponding 
Poisson bracket on $\bggo_X,$ the structure ring
of $X.$ Let $A_1$ be an Azumaya algebra over $X$ (equivalently over $\bggo_X.)$ 
Thus, we may (and will) identify the center of $A_1$ with $\bggo_X: Z(A_1)=\bggo_X.$
A deformation quantization of  $A_1$ over $W(\bold{k})$ is, by definition, a flat associatiove $W(\bf{k})$-algebra $A$ 
equipped with an isomorphism $A/pA\simeq A_1$ such that for any $a, b\in A$ such that
$a\ mod\ p\in \bggo_X, b\ mod\ p\in \bggo_X,$ one has 
$$\lbrace a\ mod\ p, b\ mod\ p\rbrace=(\frac{1}{p} [a, b])\ mod\ p.$$


 As usual,  for an associative algebra $B$ its Hochschild cohomology will be
 denoted by $HH^*(B).$  
Also, for a commutative ring $S$ over $\bold{k},$ $W_n\Omega^{*}_S$ will denote the de Rham-Witt
complex of $S$ over $W_n(\bold{k}), n\geq 1.$ 
 In the above notation, our main result is the following

\begin{theorem}\label{cohomology}

Let $A_1$ be an Azumaya algebra over an affine symplectic variety $X$ over $\bold{k}.$ Let
$A$ be a deformation quantization of $A_1$ over $W(\bold{k}).$ Then for all $n\geq 1$
we have a canonical isomorphism of graded algebras $HH^*(A/p^nA)\simeq W_n\Omega^{*}_{\bggo_X}.$

\end{theorem}

This isomorphism on the level of centers was obtained (in more general form) by Stewart and Vologodsky \cite{SV}.

Before proving the result, we will recall  Bockstein operations on   
the Hochschild cohomology of algebras over $W_m(\bold{k})$.  

Let $(S, \delta)$ be a flat differential graded algebra over $W(\bold{k}).$ 
We set $S_n=S/p^nS, n\geq 1.$
Then we have a map $v^l:S_n\to  S_{n+l}$ defined as follows: $v^l(z)=p^l\tilde{z}\ mod\ p^{n+l}S,$
where $\tilde{z}$ is a lift of $z$ in $S.$ Clearly $v^l$ is well-defined.
 Also we will denote the quotient map $ S_n\to S_{n-l}$ by $r^l.$  
Clearly $v^l, l\geq 1$ are maps of complexes, while $r^l$ is a homomorphism 
of differential graded algebras. We have the following short
exact sequence of complexes
 
 $$
\begin{CD}
0\to S_n @>v^n>> S_{2n}@> r^n>> S_n\to 0. 
\end{CD}
$$
 
Denote by $\bar{v}^l:H^{*}(S_n)\to H^{*}(S_{n+l}), \bar{r}^l:H^{*}(S_n)\to H^{*}(S_{n-l})$
 the maps induced on cohomologies by $v^l:S_n\to S_{n+l}$ and $r^l:S_n\to S_{n-l}$ respectively. 
 
 Thus we have the following long exact sequence 

$$
\begin{CD}
\cdots\to H(S_n) @>\bar{v}^n>> H(S_{2n})@> \bar{r}^n>> H(S_n) @>d_n>>H(S_n)\to\cdots,
\end{CD}
$$

\noindent here $d_n:H(S_n)\to H(S_n)$ denotes the connecting homomorphism.
 The following lemma is well-known and straightforward

\begin{lemma}\label{identities}
The algebra $(H^*(S_n), d_n)$ is a differential graded algebra in which the following identities hold
$$\bar{r}\bar{v}=\bar{v}\bar{r}=p,\quad d_n\bar{r}=p\bar{r}d_m,\quad \bar{r}d_n\bar{v}=d_n,\quad \bar{v}d_n=pd_n\bar{v}, $$
$$x\bar{v}(y)=\bar{v}(\bar{r}(x)y),\quad \bar{v}(xd_ny)=\bar{v}(x)d_n(\bar{v}y).$$

\end{lemma}

Let $A$ be an associative flat $W(\bold{k})$-algebra, not necessarily satisfying assumptions of
Theorem \ref{cohomology}. For $n\geq 1,$ set $A_n=A/p^nA.$
Recall that the standard Hochschild complex $S=(\bigoplus_n C^*(A_n, A_n), \delta)=(\Hom_{W_n(\bold{k})}(A_n^{\otimes*}, A_{n}), \delta)$
is a differential graded algebra under the cup product.
Thus we may  apply the above constructions to $S.$
Hence we have maps $$\bar{r}:HH^{*}(A_n)\to HH^{*}(A_{n-1}),\quad \bar{v}:HH^{*}(A_n)\to HH^{*}(A_{n+1})$$
$$d:HH^{*}(A_n)\to HH^{*+1}(A_n)$$
satisfying the identities from Lemma \ref{identities}.

The center of the algebra $A_n$ will be denoted by $Z_n.$
Stewart and Vologodsky [\cite{SV}, formula (1.3)] constructed a $W_n(\bold{k})$-algebra homomorphism 
 $\phi_n$ from the ring of length $n$ Witt vectors over $Z_1$ to  $Z_n, $
defined as follows: given  $(z_1, \cdots, z_n)\in W_n(Z_1), z_i\in Z_1, 1\leq i\leq n$, define

$$\phi_n(z_1, \cdots, z_n)=\sum_{i=1}^n p^{i-1}\tilde{z_i}^{p^{n-i}}$$
 where $\tilde{z_i}$ is a lift of $z_i$ in $A.$
They checked that $\phi_n$ is well-defined and $$V\phi_n=\phi_n\bar{v},\quad F\phi_n=\phi_n\bar{r},$$ where
$$V:W_n(Z_1)\to W_{n+1}(Z_1),\quad F: W_n(Z_1)\to W_{n-1}(Z_1)$$ are Verschiebung, respectively
Frobenius maps on the ring of Witt vectors of $Z_1.$

For $z\in Z_n$, denote by
$\underline{z}=\tilde{z}^p\ mod\ p^{n+1}A,  \underline{z}\in Z_{n+1},$ where $\tilde{z}$ is a lift of $z$ in $A.$ 
It was checked in \cite{SV} that $\underline{z}$ is independent of the choice of $\tilde{z}.$
Given an element $a\in Z_1$, we will denote  by $\underline{a}$  the Teichmuller lift of $a$ in $W_n(Z_1).$ Thus
$\phi_n(\underline{a})=\tilde{a}^{p^n}\in Z_n,$ where $\tilde{a}$ is a lift
of $a$ in $A_n.$

It follows from \cite[Lemma 2.5]{SV} that 
$$\bar{r}(d(\underline{z}))=z^{p-1}d(z), z\in Z_{n}.$$





The latter equality, combined with
Lemma \ref{identities} and the construction of the De Rham-Witt complex $W_n\Omega^{*}$ \cite{LZ}
 implies that 
for all $n\geq 1,$ there 
is the unique  homomorphism of
differential graded algebras $$\phi^{*}_n: W_{n}\Omega^{*}_{Z_1}\to HH^*(A_n)$$ such that
$$\bar{v}\phi^*_n=\phi^{*}_{n+1}V,\quad \bar{r}\phi^{*}_{n+1}=\phi^{*}_nF,$$ where 
$$F:W_{n}\Omega^{*}_{Z_1}\to W_{n-1}\Omega^{*}_{Z_1},\quad V:W_{n}\Omega^{*}_{Z_1}\to W_{n+1}\Omega^{*}_{Z_1}$$
denote the Frobenius, respectively Verschiebung maps in the de Rham-Witt complex of $Z_1.$
We can now prove the main result in this note.
\begin{proof}[Proof of Theorem \ref{cohomology}]
Let $A$ be a deformation  quantization of an Azumaya algebra $\tilde{A}$ over $W(\bold{k}).$
Thus we may (and will)  identify $\tilde{A}$ with $A_1=A/pA$ and $\bggo_X$ with $Z_1.$
We will prove by induction on $n$ that the map $\phi_n$ constructed above is an isomorphism. 
Let $n=1.$ Since by the assuption $A_1$ is an Azumaya algebra over $Z_1=Z(A_1)$, it follows  that
$HH^*(A_1)\simeq HH^*(Z_1)$ (\cite{Sc}). On the other hand since $\spec Z_1$ is a symplectic variety 
over $\bold{k},$ the Hochchild cohomology $HH^*(Z_1)$ is isomorphic to $\Omega^{*}_{Z_1}$ 
by the Hochschild-Kostant-Rosenberg theorem.
 Moreover, it is easy to check that given $fdg\in \Omega^1_{Z_1},$ we have
 $$\phi_1(fdg)=f\lbrace g, -\rbrace\in Der(A_1),$$ which agrees with the Hochschild-Kostant-Rosenberg
 isomorphism. Thus, $\phi_1$ is an isomorphism. 

Let $n\geq 1.$ Assume that $\phi_i$ is an isomorphism for all $i\leq n.$
We claim that the connecting homomorphism $\delta_n:HH^*(A_n)\to HH^{*+1}(A_1)$
between Hochschild cohomologies
arising from the exact sequence  of $A_{n+1}$-bimodules 
$$
\begin{CD}
0\to A_1 @>v^n>> A_{n+1} @> r>> A_n\to 0
\end{CD}
$$

\noindent equals ${\bar{r}}^nd.$ Indeed, it follows directly from commutativity of the following diagram 
$$\begin{CD}
HH^*(A_{n}) @>{\bar{v}}^n>> HH^*(A_{2n}) @>\bar{r}^n>> HH^*(A_n) @>d>> HH^{*+1}(A_n)\\
@VV{\bar{r}}^{n-1}V  @VV{\bar{r}}^{n-1}V    @VVidV      @VV{\bar{r}}^{n-1}V\\
HH^{*}(A_1) @>{\bar{v}}^n>> HH^{*}(A_{n+1}) @>\bar{r}>> HH^{*}(A_{n})@>\delta_n>> HH^{*+1}(A_1).\\
\end{CD}$$

We have the following commutative diagram of long exact sequences
$$\begin{CD}
W_1\Omega^{*}_{Z_1} @>V^n>> W_{n+1}\Omega^{*}_{Z_1} @>F>> W_{n}\Omega^{*}_{Z_1} @>F^{n-1}d>> W_{1}\Omega^{*+1}_{Z_1}\\
@VV\phi_1V  @VV\phi_{n+1}V    @VV\phi_nV      @VV\phi_1V\\
HH^{*}(A_1) @>\bar{v}^n>> HH^{*}(A_{n+1}) @>\bar{r}>> HH^{*}(A_{n})@>{\bar{r}}^{n-1}d>> HH^{*+1}(A_1)\\
\end{CD}$$
Here exactness of the top sequence is a well-known property of the de Rham-Witt complex
\cite[Proposition 3.11]{Il}.

Thus, it follows from the inductive assumption that $\phi_{n+1}$ is an isomorphism.
This concludes the proof of Theorem \ref{cohomology}.

\end{proof}

Given a smooth affine variety $Y$ over a ring $S$,  we will denote the
ring of crystalline (or PD) differential operators on $Y$ by $D_Y=D_{Y/S}$.

\begin{cor}
Let $Y$ be a smooth affine variety over $W(\bold{k}).$ For $n\geq 1,$ let  $Y_n$  denote the $\mod$
$p^n$ reduction of  $Y.$ Then $HH^*(D_{Y_n})$ is isomorphic to $W_n\Omega^{*}_{T^{*}_{Y_1}},$
where $T^{*}_{Y_1}$ is the cotangent bundle of $Y_1.$

\end{cor}

\begin{proof}

Put $A=D_{\tilde{Y}}.$
Then $A_n=A/p^nA=D_{Y_n}, n\geq 1.$
It is well-known that $\tilde{A}=D_{Y_1}$ is an Azumaya algebra over the Frobenius twist
of $T^{*}_{Y_1}=X$-the cotangent bundle of $Y_1 $\cite{BMR}. Also, the corresponding  Poisson bracket
on $X$ is symplectic by \cite{BK}. Thus, Theorem \ref{cohomology} applies directly.

\end{proof}

Next we will establish a generalization of a result by Stewart and Vologodsky [\cite{SV}, Theorem 1].
We will follow their proof very closely.

\begin{theorem}\label{iso}

Let $(X,\lbrace, \rbrace)$ be an affine Poisson normal variety over $\bf{k}$ such that it has an open symplectic leaf.
Let $A$ be a flat associative $W_n(\bf{k})$-algebra, such that $Z(A/pA)=\bggo_X$ and the deformation
Poisson bracket of $Z(A/pA)$ coincides with the bracket of Poisson variety $X.$ Then the map $\phi_m:W_m(\bggo_X)\to Z(A/p^mA)$ 
defined above is an isomorphism for all $1\leq m\leq n.$

\end{theorem}

\begin{proof}
We will proceed by induction on $m.$ For $m=1$ there is nothing to prove. As before, we will put
$A_m=A/p^mA, Z_m=Z(A_m).$ Recall also the isomorphisms $v^i:A_m\to p^iA_{m+i}$ defined by multiplying a lift in
$A_{m+i}$ by $p^i, i\leq n-m.$
The following is the key
step (as in \cite{SV}, Lemma 2.4).
\begin{lemma}

Let $\tau:W_m(Z_1)\to Der(Z_1, Z_1)$ be defined as follows $$\tau(z_1,\cdots,z_m)=\sum_{i=1}^mz_i^{p^{m-i}-1}\lbrace z_i,-\rbrace.$$
Then $\tau(z_1,\cdots,z_m)=0$ if and only if $z_i=a_i^p$ for some $a_i\in Z_1.$

\end{lemma}

\begin{proof}
Let $f\in Z_1$ be such that $U=\spec (Z_1)_f$ is an open symplectic leaf of $X.$
Put $S=(Z_1)_f$, let $\omega$ be the symplectic form of $U.$ Transferring the equality
$\tau(z_1,\cdots,z_n)=0$ via the isomorphism $\omega:Der_{\bf{k}}(S,S)\simeq \Omega^1_S$ to $\Omega^1_S,$ we obtain $$\sum_{i=1}^mz_i^{p^{m-i}-1}dz_i=0.$$
Recall that the inverse Cartier map $C^{-1}:\Omega^1_S\to\Omega^1_S/dS,$
which is defined as follows $C^{-1}(fdg)=f^pg^{p-1}dg,$ is an injective homomorphism onto
$H^1(\Omega^{*}_S)\subset \Omega^1_S/dS.$ Thus, we obtain
$$C^{-1}(\sum_{i<m}z_i^{p^{m-i}-1}dz_i)=0\in \Omega^1/dS$$
Hence by injectivity of $C^{-1}$ we conclude that $$\sum_{i=1}^{m-1}z_i^{p^{m-1-i}-1}dz_i=0\in \Omega^1_S.$$
Using induction assumption, we conclude that $z_i\in S^p, i>1.$ Thus $dz_1=0$, so $z_1\in S^p.$
So, for each $i, z_i\in Z_1\cap S^p.$ 
Since by the assumption $Z_1$ is a normal domain, 
we have that $Z_1\cap S^p=(Z_1)^p.$
Hence $z_i\in Z_1^p$ and we are done.

\end{proof}
Now we can prove the theorem. Let $x\in Z_m.$ Then by the inductive assumption 
$x\mod p^{m-1}=\phi_{m-1}(z)$ 
for some $z\in W_{m-1}(Z_1).$
It follows from computations in [\cite{SV}, proof of Lemma 2.5] that
$$\tau(z)=(\frac{1}{p^{m-1}}[\phi_{m-1}(z),-]\mod p)|_{Z_1}=(\frac{1}{p^{m-1}}[x,- ]\mod p)|_{Z_1}=0.$$ 
So by the above lemma, $z=F(z')$, for some
$z'\in W_m(Z_1)$. Hence we conclude that $\phi_n(z')-x\in p^{m-1}A_m.$ So $Z_m\subset \phi_m(W_m(Z_1))+p^{m-1}A_m.$
But $$Z_m\cap p^{m-1}A_m=v^{m-1}(Z_1)\subset \phi_m(W_m(Z_1)).$$
 Therefore $Z_m=\phi_m(W_m(Z_1)).$ 

It remains to check that $\phi_m$ is injective. Let $0\neq x\in W_{m}(Z_1)$. Hence we may write
$x=V^i(z)$, for some i and $z\in W_{m-i}(Z_1)\setminus VW_{m-i-1}(Z_1).$
We have 
 $$\phi_m(x)=\phi_m(V^i(z))=v^i(\phi_{m-i}(z)).$$
Clearly $\phi_i(z)\in A_{m-i}\setminus pA_{m-i}.$ Hence $\phi_m(V^i(z))\neq 0.$
\end{proof}

The above result can be used to compute centers of large class of algebras over $W_{n}(\bf{k}),$
including various enveloping algebras of Lie algebras and symplectic reflection algebras.

\end{document}